\documentclass[preprint,3p]{elsarticle}
\usepackage{amsthm,amsmath,amssymb}
\newtheorem{defi}{Definition}[section]

\newtheorem{theorem}[defi]{Theorem}
\newtheorem{cor}[defi]{Corollary}

\journal{Commun Nonlinear Sci Numer Simulat}

\begin{document}

\begin{frontmatter}

\author[1]{Mehdi Nadjafikhah\corref{cor1}}
\ead{m\underline{ }nadjafikhah@iust.ac.ir}
\author[2]{Mehdi Jafari}
\ead{m.jafari@phd.pnu.ac.ir}

\cortext[cor1]{Corresponding author. Tel.:+98 21 73225426; Fax: +98 21 73228426}

\address[1]{School of Mathematics, Iran University of Science and Technology, Narmak, Tehran 1684613114, Iran}
\address[2]{Department of Complementary Education, Payame Noor University, PO BOX 19395-3697, Tehran, Iran}

\title{Computation of Partially Invariant Solutions for the Einstein Walker Manifolds' Identifying Equations}

\begin{abstract}
In this paper, partially invariant solutions (PISs) method is
applied in order to obtain new four-dimensional Einstein Walker
manifolds. This method is based on subgroup classification for
the symmetry group of partial differential equations (PDEs) and
can be regarded as the generalization of the similarity reduction
method. For this purpose, those cases of PISs which have the
defect structure $\delta=1$ and are resulted from two-dimensional
subalgebras are considered in the present paper. Also it is shown
that the obtained PISs are distinct from the invariant solutions that
obtained by similarity reduction method.
\end{abstract}


\end{frontmatter}


\section{Introduction}
\label{}

The idea of analyzing the differential equations by applying the
transformation groups implied a new theory: the symmetry group
theory, also called Lie group analysis. This method was originated
at the end of nineteenth century from the pioneering work of
Sophus Lie \cite{Lie}. Since that time, several books have been
dedicated to this concept and its generalizations [2-5]. Classification of the group invariant
solutions and reduction of the original system can be regarded as
significant applications of the Lie group method in the theory of
differential equations. The fact that symmetry reductions for
many PDEs can not be obtained via the classical symmetry method,
motivated the creation of  several generalizations of the
classical Lie group method for symmetry reductions. Consequently,
several alternative reduction methods have been proposed, going
beyond Lie's classical procedure and providing further solutions.
Partially invariant solutions (PISs) method is one of these
procedures. This algorithmic method is a powerful tool for the
reduction of PDEs and is based on classifying the subgroups of
the symmetry group. The notion of PISs  was first developed by
Ovsiannikov \cite{Ovs} and can be considered as the extension of
invariant solutions. The algorithm of constructing PISs is
similar to that of invariant solutions. Indeed, obtaining the
invariant solutions by applying the PISs method is  easier than
similarity reduction method whenever we deal with low-dimensional
groups. One significant concept which appears while constructing
PISs is the defect structure. This quantity is determined by the
dimension of
orbits and is denoted by $\delta$.\\
$~~$ In this paper,  the PISs method will be applied in order to
construct some new classes of four-dimensional Einstein Walker
manifolds. Those Manifolds which admit null parallel distributions
are called {\it Walker manifolds}. A Walker manifold is called
{\it Einstein Walker manifold} if  its Ricci tensor is a scaler
multiple of the metric at each point. Four-dimensional Einstein Walker
manifolds form the underling structure of many geometric and physical models such as; hh-space in general relativity, pp-wave  model and other areas [6-12].

 The general form of the metric tensor of four-dimensional walker manifolds is expressed as
follows \cite{Gar}:
\begin{eqnarray}
 \begin{array}{lclcl}\label{1}
 g_{a,b,c}:=2(dx\circ dy+dt\circ  dz)+a(x,t,y,z)dy\circ dy\vspace*{1mm}\\
 \hspace{13mm}+b(x,t,y,z)dz\circ dz+2c(x,t,y,z)dy\circ dz,
 \end{array}
 \end{eqnarray}
where $a$, $b$ and $c$ are smooth functions with respect to
$(x,t,y,z)$. Let $M_{a,b,c}:=(\mathcal{O},g_{a,b,c})$, where $\mathcal{O}$ be an open subset of $\mathbb{R}^4$. We can see that $M_{a,b,c}$ is Einstein if and only if the functions $a$, $b$ and $c$ verify the  following system of PDEs (\cite{Gar}, page 81).
\begin{eqnarray}\label{a7}
\begin{array}{lclclc}
a_{11}-b_{22}=0,\hspace{12mm} b_{12}+c_{11}=0,\hspace{12mm}a_{12}+c_{22}=0,\\
a_{1}c_{2}+a_{2}b_{2}-a_{2}c_{1}-{c_{2}}^{2}+2c_{}a_
{12}+b_{}a_{22}-2a_{24}-a_{}c_{12}+2c_{23}=0,\\
a_{2}b_{1}-c_{1}c_{2}+c_{}a_{11}-a_{14}-b_{23}-a_{}c_{11}-c_{}c_{12}+c_{13}-b_{}c_{22}+c_{24}=0,\\
a_{1}b_{1}-b_{1}c_{2}+b_{2}c_{1}-{c_{1}}^{2}+a_{}b_{11}+2c_{}b_{12}-2b_{13}-b_{}c_{12}+2c_{14}=0,
 \end{array}
  \end{eqnarray}
 where the index 1, 2, 3 and 4 for functions $a$, $b$ and $c$ represent
 the derivatives of these functions with respect to $x$, $t$, $y$ and $z$, respectively.
The system (\ref{a7}) is hard to handle, so we consider a spacial case in this paper;
 where $a$, $b$ and $c$ only depend on $x$ and $t$. Therefore the following system must be solved.
\begin{eqnarray}\label{2}
\begin{array}{lclclc}
a_{11}-b_{22}=0,\hspace*{7mm} b_{12}+c_{11}=0,\hspace*{7mm}a_{12}+c_{22}=0,\\
a_{1}c_{2}+a_{2}b_{2}-a_{2}c_{1}-{c_{2}}^{2}+2c_{}a_
{12}+b_{}a_{22}-a_{}c_{12}=0,\\

a_{2}b_{1}-c_{1}c_{2}+c_{}a_{11}-a_{}c_{11}-c_{}c_{12}-b_{}c_{22}=0,
\\
a_{1}b_{1}-b_{1}c_{2}+b_{2}c_{1}-{c_{1}}^{2}+a_{}b_{11}+2c_{}b_{12}-b_{}c_{12}=0.
 \end{array}
 \end{eqnarray}
In system (\ref{2}), $x$ and $t$ are independent and $a$, $b$ and $c$ are dependent variables. It is worthwhile to say that some other special cases have been considered in many references which yield some results about the structures admitted by these manifolds [14-16]. \\
$~~$ In \cite{Nad}, we have comprehensively analyzed the problem
of symmetries of the system (\ref{2}). By applying the basic Lie
symmetry method, we have obtained the classical Lie point
symmetry operators of the system (\ref{2}) and proved the following
result (refer to \cite{Nad} for more details):
\begin{cor}\label{1.1}
The Lie group of point symmetries of the PDE system (\ref{2}) has a
seven-dimensional Lie subalgebra generated by the following vector
fields:
 \begin{eqnarray}\label{3}
\begin{array}{lclclclclc}
X_1=\partial_x,&& X_2=\partial_t,&& X_3=x\partial_x-2b\partial_b-c\partial_c,\\
 X_4=x\partial_t+2c\partial_b+a\partial_c,&& X_5=t\partial_x+2c\partial_a+b\partial_c,&& X_6=t\partial_t+2b\partial_b+c\partial_c,\\X_7=a\partial_a+b\partial_b+c\partial_c,
\end{array}
\end{eqnarray}
($\partial_x\equiv\frac{\partial}{\partial x}$,...).
\end{cor}
Mainly, we have constructed an optimal system of one-dimensional
subalgebras in \cite{Nad} which provides the preliminary
classification of group invariant solutions for the system (\ref{2}).
Also, we have obtained the corresponding invariant solutions of
this system via the method of similarity reduction.\\
It is worth considering that some of the partially invariant
solutions are not invariant with respect to the subalgebras of
lower dimensions. They are known as non-reducible PISs. By
determining non-reducible  PISs for the system (\ref{2}), we can obtain
some new
four-dimensional Einstein Walker manifolds.\\
This paper is organized as follows: In section 2, we recall the
general procedure of determining PISs. In section 3, we classify
the two-dimensional subalgebras of the symmetry Lie algebra and
construct an optimal system. Section 4 is devoted to the
computation of some of the PISs associated to the system (\ref{2}). In
section 5, the relation between PISs and invariant solutions is
investigated and the condition for obtaining the non-reducible
PISs is stated. Meanwhile, some non-reducible PISs for the system
(\ref{2}) are presented. Some concluding remarks are declared at the end
of the paper.
\section{Partially invariant solutions method}
In this section, we recall the general procedure for determining
PISs for an arbitrary system of PDEs. First, we present a brief
review of the concept of PISs  \cite{Melesh,Ovs}. Consider a
system of PDEs of $n$th order  with $p$ independent variables
$(x=(x^{i})\in X, i=1,...,p)$ and $q$ dependent variables
($u=(u^{j})\in U, j=1,...,q$) such as:
\begin{eqnarray}\label{4}
\Delta=\Delta_{\mu}(x,u^{(n)})=0,    \ \ \ \ \ \ \ \ \          \mu=1,...,r.
\end{eqnarray}
Let $G$ be a local symmetry group of the above system which acts
on the total space $X\times U$ with $r$-dimensional orbits. If
$u=f(x)$ is a solution of the system (\ref{4}) with graph $\Gamma_f$,
then the orbit space of $\Gamma_f$ can be defined as follows:
\begin{eqnarray}\label{5}
G\Gamma_f=\{g.(x,u)\mid (x,u)\in \Gamma_f, g\in G\}
\end{eqnarray}
which is the union of the orbits of the $\Gamma_f$-elements.\\
 The
defect structure of the solution $u=f(x)$ with respect to the
group $G$ is computed by the matrix of generators'
characteristics and is defined as:
\begin{eqnarray}\label{6}
\delta=\mathrm{dim}(G\Gamma_f)-\mathrm{dim}(\Gamma_f)=\mathrm{dim}(G\Gamma_f)-p.
\end{eqnarray} Also, we have $0\leq\delta\leq
\mathrm{min}\{r,q\}$ (\cite{Ovs}, p 276-277).\\
 If $\delta =0$,
then $u=f(x)$  is an invariant solution and if $0<\delta<
\mathrm{min}\{r,q\}$, then $u=f(x)$ is a partially invariant
solution.

In order to calculate the PISs, first of all, it is necessary to
classify the symmetry group into conjugacy classes. For obtaining
the PISs with the defect structure $\delta$, those subgroups
$H\subset G$ which have the property that if the dimension of the
 orbits of $H$ on the space $X\times U$ is $r$, then the dimension of
the orbits is $r-\delta$ on the space $X$,  must be selected
\cite{Grun}. Let $H$ be a subgroup with this property that
mentioned above and $\mathfrak{h}$ be its Lie algebra with
infinitesimal generators $\{v_1,...,v_s\}$. Hence, we can obtain a
complete set of functionally independent invariants of the form
\begin{eqnarray}\label{7}
\{\xi_i(x),I_j(x,u)\},
\end{eqnarray}
where $i=1,...,p+\delta-s$ and $j=1,...,q-\delta$.\\
 Then we have
\begin{eqnarray}\label{8}
\mathrm{rank}\left(\frac{\partial I_j(x,u)}{\partial u}\right)
=q-\delta=q'.
\end{eqnarray}
If $u=f(x)$ is a function, then the manifold $H\Gamma_f$ can be
expressed in terms of the invariants (\ref{7}). So, we have:
\begin{eqnarray}\label{9}
I_j(x,u)=f_j(\xi_i(x))
\end{eqnarray}
where the functions $f_j$ are arbitrary. Now, by applying the
implicit function theorem, we conclude that
\begin{eqnarray}\label{10}
u^{i_\alpha}=U^{i_\alpha}(x,u^{j\beta},f_j(\xi_i(x)))
\end{eqnarray}
where $\alpha=1,...,q'$ and $\beta=1,...,\delta$. The remaining
dependent variables only depend on the original independent
variables:
\begin{eqnarray}\label{11}
u^{j_\beta}=U^{j_\beta}(x_1,...,x_p), \ \ \ \ \ \ \ \ \beta=1,...,\delta.
\end{eqnarray}
Now, the derivatives of the functions $u^1,...,u^q$ with respect
to the new variables which are obtained from equations (\ref{10}) and (\ref{11})
must be calculated. Hence, by  substituting these quantities into
the original system,  a new system is obtained, involving the $q'$
functions $f_j$ and the invariants $\xi_i$. The resulted equations
are not generally consistent, so  that the compatibility
conditions must be computed. Consequently, a system of PDEs is
deduced from these constraints which is denoted by $\Delta/H$. On
the other hand, a system of PDEs is resulted from (\ref{11}) denoted by
$\Delta'$. Now, the system  $\Delta/H$ must be solved first. Then
corresponding to each of the solution of this system, the system
$\Delta'$ will be solved. Finally, the partially invariant
solutions are obtained by substituting the resulted solutions
into equations (\ref{10}) and (\ref{11}).  For more details about the methods of
determining the PISs refer to \cite{Grun2,Ovs}.
\section{Classification of Subalgebras for the system (\ref{2})}
In this section, we want to classify the subgroups of symmetry
group of the system (\ref{2}), into conjugacy classes. Searching for the invariant solutions can be regarded as the main
motivation of computing the symmetries of a differential
equation. As it is well known, the problem of classifying
invariant solutions is equivalent to the problem of classifying
the subgroups of the full symmetry group under conjugation. Let
$H$ and $\tilde{H}$ be two connected, $s$-dimensional Lie
subgroups of the Lie group $G$ with corresponding Lie subalgebras
${\mathfrak h}$ and $\tilde{\mathfrak h}$ of the Lie algebra
${\mathfrak g}$ of $G$. Let $g\in G$, then $\tilde{H}=gHg^{-1}$
are conjugate subgroups if and only if $\tilde{\mathfrak h}={\rm
Ad}(g)\cdot{\mathfrak h}$ are conjugate subalgebras, where ${\rm Ad}(g)$
is adjoint represen Hence, the
problem of determining  an optimal system of subgroups is
equivalent to that of obtaining an optimal system of subalgebras,
and so we focus on it \cite{Olv,Ovs}. The latter problem, tends to
obtain a list of conjugacy inequivalent subalgebras with the
property that any other subalgebra is equivalent to only a unique
member of the list under some element of the adjoint
representation for some element of the investigated Lie group.

\subsection{Optimal system of one-dimensional subalgebras for the
system (\ref{2})}
Indeed, for one-dimensional subalgebras, the classification
problem is necessarily the same as the problem of classifying the
orbits of the adjoint representation. Thus, an optimal set of
subalgebras is constructed if we select just one representative
from each family of equivalent subalgebras. Consequently, the
associated set of invariant solutions is then the minimal list
from which we can obtain all other invariant solutions of
one-dimensional subalgebras simply via transformations.

In \cite{Nad}, we have presented a comprehensive analysis of this
problem and have constructed an optimal system of one-dimensional
subalgebras for the system (\ref{2}) as follows :
\begin{theorem}\label{3.1}
An optimal system of one-dimensional Lie subalgebras of the
system (\ref{2}) is provided by the following generators:
\begin{eqnarray}\label{12}
\begin{array}{lclcl}
1)\ {\rm{\bf  X}}^1=X_7,&&\hspace*{2mm}8)\ {\rm{\bf  X}}^{8}\hspace{1.5mm}=X_4+aX_5+bX_6+cX_7,\vspace*{.5mm}\\
2)\ {\rm{\bf  X}}^2=X_1+aX_7,&&\hspace*{2mm}9)\ {\rm{\bf  X}}^{9}\hspace{1.5mm}=\varepsilon X_1+X_4+aX_5+bX_6+cX_7,\vspace*{.5mm}\\
3)\ {\rm{\bf  X}}^3=X_2+aX_7,&&10)\ {\rm{\bf  X}}^{10}=X_3+aX_5+bX_6+cX_7,\vspace*{.5mm}\\
4)\ {\rm{\bf  X}}^4=X_6+aX_7\,&&11)\ {\rm{\bf  X}}^{11}=\varepsilon X_2+X_3+aX_5+bX_6+cX_7,\vspace*{.5mm}\\
5)\ {\rm{\bf  X}}^5=\varepsilon X_1+X_6+aX_7,&&12)\ {\rm{\bf  X}}^{12}=X_3+\varepsilon X_4+aX_5+bX_6+cX_7,\vspace*{.5mm}\\
6)\ {\rm{\bf  X}}^6=X_5+aX_6+bX_7,&&13)\ {\rm{\bf  X}}^{13}=\varepsilon X_2+X_3+\varepsilon' X_4+aX_5+bX_6+cX_7.\vspace*{.5mm}\\
7)\ {\rm{\bf  X}}^7=\varepsilon X_2+X_5+aX_6+bX_7,
\end{array}
\end{eqnarray}
where $\varepsilon$ and $\varepsilon'$ are $\pm 1$  and $a,b,c\in{\Bbb R}$ are arbitrary numbers \mbox{\cite{Nad}}.
\end{theorem}
\subsection{Optimal system of two-dimensional subalgebras for the
system (\ref{2})} In this paper, we need to classify the two-dimensional
subalgebras. Because, we want to calculate those PISs which have
the defect structure $\delta=1$ and the reduced system $\Delta/H$
is a system of ordinary differential equations. Since $p=2$,
$\delta=1$ and $p+\delta-s=1$, then we have $s=2$. So, we should
consider the two-dimensional subgroups. \\
Consequently,
constructing the two-dimensional optimal system, i.e.,
classification of the two-dimensional subalgebras of $\mathfrak{g}$
is our next step. This process is performed by choosing one of the
vector fields as stated in theorem (\ref{3.1}). Let us consider
${\rm{\bf  X}}^1$ (or ${\rm{\bf X}}^i,\ i=2,\cdots,13$).
Corresponding to it, an optional vector field
$Y=b_1X_1+\cdots+b_7X_7$ is selected, so we must have
\begin{eqnarray}\label{13}
[{\rm{\bf  X}}^1,Y]=\lambda {\rm{\bf  X}}^1+\mu Y.
\end{eqnarray}
Equation (\ref{13}) leads us to the system
\begin{eqnarray}\label{14}
C^i_{jk}\alpha_ja_k=\lambda
a_i+\mu\alpha_i\hspace{2cm}(i=1,\cdots,7).
\end{eqnarray}
where the constant coefficients $C^i_{jk}$ are the structure constants. The solutions of the system (\ref{14}), give one of the
two-dimensional generators and the second generator is ${\rm{\bf
X}}^1$ (or ${\rm{\bf X}}^i,\ i=2,\cdots,13$) if selected.\\
 After
the construction of all two-dimensional subalgebras, for every
vector fields of theorem (\ref{3.1}), they need to be simplified by the
action of adjoint transformations in the manner analogous to the
way of one-dimensional optimal system. Hence, we can state the
following theorem:

\begin{theorem}\label{3.2}
Any two-dimensional subalgebras of (\ref{3}) is conjugate to precisely
one of the following subalgebras:

\begin{eqnarray*}
\begin{array}{lclclclclcl}
\mathcal{A}^1_1 : \langle X_1,X_3+\alpha X_6+\beta X_7\rangle,&&\hspace{-3mm} \mathcal{A}^2_1 : \langle X_1,X_2+\alpha  X_5+\beta X_7\rangle,&&\hspace{-3mm}\mathcal{A}^3_1 : \langle X_1,X_5+\alpha X_7\rangle,\\
\mathcal{A}^4_1 : \langle X_1,X_3+\epsilon X_5+X_6+\alpha X_7\rangle,&&\hspace{-3mm}\mathcal{A}^5_1 : \langle X_1,X_2+\alpha X_3+\beta X_7\rangle
, &&\hspace{-3mm}\mathcal{A}^6_1 : \langle X_1,X_6+\alpha X_7\rangle,\\
\mathcal{A}^7_1 : \langle X_1,X_7\rangle,
\vspace*{3mm}\\

\mathcal{A}^1_2 : \langle X_2,X_3+\alpha X_6+\beta X_7\rangle,&&\hspace{-3mm} \mathcal{A}^2_2 : \langle X_2,X_1+\epsilon  X_4+\beta X_7\rangle,&&\hspace{-3mm}\mathcal{A}^3_2 : \langle X_2,X_4+\alpha X_7\rangle,\\
 \mathcal{A}^4_2 : \langle X_2,X_3+\epsilon X_4+X_6+\alpha X_7\rangle,&&\hspace{-3mm} \mathcal{A}^5_2 : \langle X_2,X_1+X_6+\alpha X_7\rangle,&&\hspace{-3mm}\mathcal{A}^6_2 : \langle X_2,X_6+\alpha X_7\rangle, \\
 \mathcal{A}^7_2 : \langle X_2,X_7\rangle,
\vspace*{3mm}\\

\mathcal{A}^1_3 : \langle X_6,X_3+\alpha X_7\rangle,&&\hspace{-3mm}\mathcal{A}^2_3 : \langle X_6,X_4\rangle,&&\hspace{-3mm} \mathcal{A}^3_3 : \langle X_6,X_5\rangle,\\
\mathcal{A}^4_3 : \langle X_6,X_1+\alpha X_7\rangle,&&\hspace{-3mm}\mathcal{A}^5_3 : \langle X_6,X_2\rangle,&&\hspace{-3mm}\mathcal{A}^6_3 :
\langle X_6,X_7\rangle,
\vspace*{3mm}\\

\mathcal{A}^1_4 : \langle \varepsilon X_1+X_6,X_2\rangle ,&&\hspace{-3mm}\mathcal{A}^2_4 : \langle \varepsilon X_1+X_6,X_5\rangle, &&\hspace{-3mm} \mathcal{A}^3_4 : \langle \varepsilon X_1+X_6,X_7\rangle,
\vspace*{3mm}\\

\mathcal{A}^1_5 : \langle X_5,X_3+\alpha X_6+\beta X_7\rangle ,&&\hspace{-3mm} \mathcal{A}^2_5 : \langle X_5,X_1+\alpha X_6+\beta X_7\rangle,&&\hspace{-3mm}\mathcal{A}^3_5 : \langle X_5,X_6+\alpha X_7\rangle,\\
\mathcal{A}^4_5 : \langle X_5,X_7\rangle,
\vspace*{3mm}\\

\mathcal{A}^1_6 : \langle \varepsilon X_2+X_5,X_3+\frac{1}{2}X_6+\alpha X_7\rangle ,&&\hspace{-3mm} \mathcal{A}^2_6 : \langle \varepsilon X_2+X_5,X_1+\alpha X_7\rangle,&&\hspace{-3mm}
\mathcal{A}^3_6 : \langle \varepsilon X_2+X_5,X_7\rangle,
\vspace*{3mm}\\

\mathcal{A}^1_7 : \langle X_4,X_3+\alpha X_6+\beta X_7\rangle ,&&\hspace{-3mm} \mathcal{A}^2_7 : \langle X_4,X_2+\alpha X_3+\beta X_7\rangle,&&\hspace{-3mm} \mathcal{A}^3_7 : \langle X_4,X_6+\alpha  X_7\rangle,\\
\mathcal{A}^4_7 : \langle X_4,X_7\rangle,
\vspace*{3mm}\\

\mathcal{A}^1_{8} : \langle \varepsilon X_1+X_4,X_3+2X_6+\alpha X_7\rangle ,&&\hspace{-3mm}\mathcal{A}^2_{8} : \langle \varepsilon X_1+X_4,X_2+\alpha X_7\rangle,&&\hspace{-3mm} \mathcal{A}^3_{8} : \langle \varepsilon X_1+X_4,X_7\rangle,
\vspace*{3mm}\\

\mathcal{A}^1_{9} : \langle X_3,X_2+\alpha X_7\rangle,&&\hspace{-3mm}\mathcal{A}^2_{9} : \langle X_3,X_5\rangle,&&\hspace{-3mm}\mathcal{A}^3_{9} : \langle X_3,X_1\rangle,
\\
\mathcal{A}^4_{9} : \langle X_3,X_6+\alpha X_7\rangle,&&\hspace{-3mm}
\mathcal{A}^5_{9} : \langle X_3,X_7\rangle,&&\hspace{-3mm} \mathcal{A}^6_{9} : \langle X_3,X_4\rangle,
\vspace*{3mm}\\

\mathcal{A}^1_{10} : \langle \varepsilon X_2+X_3,X_1\rangle ,&&\hspace{-3mm} \mathcal{A}^2_{10} : \langle \varepsilon X_2+X_3,X_4\rangle,&&\hspace{-3mm}\mathcal{A}^3_{10} : \langle \varepsilon X_2+X_3,X_7\rangle,\\
\mathcal{A}^4_{10} : \langle X_2,X_3+\alpha X_7\rangle,

\vspace*{3mm}\\

\mathcal{A}^1_{11} : \langle X_3+\varepsilon X_4,\varepsilon X_5+X_6-2X_7\rangle ,&&\hspace{-3mm} \mathcal{A}^2_{11} : \langle X_3+\varepsilon X_4,X_2+\alpha X_7\rangle,&&\hspace{-3mm}\mathcal{A}^3_{11} : \langle X_3+\varepsilon X_4,X_7\rangle,\\
\mathcal{A}^4_{11} : \langle X_3+\varepsilon X_4,X_3+X_6+\alpha X_7\rangle,&&\hspace{-3mm} \mathcal{A}^5_{11} : \langle X_3+\varepsilon X_4,X_1+\varepsilon X_2\rangle,
\vspace*{3mm}\\

\mathcal{A}^1_{12} : \langle \varepsilon X_2+X_3+\varepsilon' X_4,X_1+\varepsilon' X_2\rangle ,&&\hspace{-3mm}
\mathcal{A}^2_{12} : \langle \varepsilon X_2+X_3+\varepsilon' X_4,X_7\rangle,
\end{array}
\end{eqnarray*}
where $\alpha$ and $\beta$ are arbitrary constants,
$\varepsilon$ and $\varepsilon'$ are $\pm 1$ and $\epsilon$ is $\pm 1$ or $0$.
\end{theorem}
\begin{proof}
Each of the two-dimensional subalgebra has two generators. For
classifying two-dimensional subalgebras, we must select one of the
generators from the list of one-dimensional optimal system (\ref{12})
and another generator must be taken optionally.
 \\
Suppose that $\mathfrak{h}=\mathrm{span}\{X,Y\}$ is a
two-dimensional subalgebra of $\mathfrak{g}$ where $X$ is a
one-dimensional subalgebra which is chosen from the list (\ref{12}) and
$Y$ is an optional vector defined by: $Y=b_1X_1+\cdots+b_7X_7$.
Now, we must simplify $\mathfrak{h}$ as much as possible by
imposing various adjoint transformations on it \cite{Olv}.\\
Each adjoint transformation is a linear map $F^s_i:\mathfrak{g}\to
\mathfrak{g}$ defined by $X\mapsto\mathrm{Ad}(\exp(sX_i).X)$, for
$i=1,\cdots,7$. Since this procedure is so lengthy,  we only
explain one of the cases in the following. \vspace*{1mm}

Case1 : If $X=X_1+aX_7$ then
\begin{eqnarray}\label{15}
\begin{array}{lclcl}
\mathfrak{h}=\langle X,Y\rangle=\langle X_1+aX_7,\sum^7_{i=1}b_iX_i\rangle\vspace*{2mm}\\
\hspace{19.3mm}=\langle
X_1,b_2X_2+b_3X_3+b_4X_4+b_5X_5+b_6X_6+b_7X_7\rangle
\end{array}
\end{eqnarray}
So, we have:
\begin{itemize}
\item[a)] If $b_2=b_3=b_4=b_5=b_6=0$, then we have $\mathfrak{h}= \langle  X_1,b_7X_7 \rangle=\langle X_1,X_7\rangle$. Since $[X_1,X_7]=0=0X_1+0X_7$, then  $\mathfrak{h}$ is closed under the Lie bracket. So, $\mathfrak{h}$ is reduced to the case  $\mathcal{A}^7_1$.
\item[b)] If $b_2=b_3=b_4=b_5=0$ and $b_6\neq 0$, then $\mathfrak{h}= \langle  X_1,b_6X_6+b_7X_7 \rangle=\langle X_1,X_6+\alpha X_7\rangle$. Since $[X_1,X_6+\alpha X_7]=0$, then  $\mathfrak{h}$ is closed under the Lie bracket. So, $\mathfrak{h}$ is reduced to the case $\mathcal{A}^6_1$.
\item[c)] If $b_2=b_3=b_4=b_6=0$ and $b_5\neq 0$, then $\mathfrak{h}= \langle  X_1,b_5X_5+b_7X_7 \rangle=\langle X_1,X_5+\alpha X_7\rangle$. Since $[X_1,X_5+\alpha X_7]=0$, then  $\mathfrak{h}$ is closed under the Lie bracket. So, $\mathfrak{h}$ is reduced to the case  $\mathcal{A}^3_1$.
\item[d)] If $b_2=b_3=b_4=0$ and $b_6\neq 0$, then we can make the coefficient of $X_5$ vanish by $F^{s_5}_5$; By setting $s_5=\frac{1}{b_6}$. So, we  have $\mathfrak{h}= \langle  X_1,b_6X_6+b_7X_7 \rangle$ which is reduced to the case  $\mathcal{A}^6_1$ like (b).
\item[e)] If $b_2=b_3=0$ and $b_4\neq 0$, then $\mathfrak{h}=\langle X_1,b_4X_4+\cdots+b_7X_7 \rangle$ and  $[X_1,b_4X_4+\cdots+b_7X_7]=b_4X_2\neq rX_1+s(b_4X_4+\cdots+b_7X_7)$ for any $r,s\in\mathbb{R}$.  So, $\mathfrak{h}$ is not closed under the Lie bracket and we have not any two-dimensional subalgebras in this case.
\item[f)] If $b_2=b_4=0$, $b_3\neq 0$ and $b_6\neq 1$, then we can make the coefficient of $X_5$ vanish by
 $F^{s_5}_5$; By setting $s_5=\frac{b_5}{-1+b_6}$. Also, by scaling if necessary,
 we can assume that $b_3=1$. Then we have  $\mathfrak{h}=\langle X_1,X_3+b_6X_6+b_7X_7 \rangle$ and $[X_1,X_3+b_6X_6+b_7X_7]=X_1$.
 So, $\mathfrak{h}$ is closed under the Lie bracket and the case $\mathcal{A}^1_1$ is concluded.
\item[g)] If $b_2=b_4=0$, $b_3\neq 0$ and $b_6=1$, then by scaling if necessary, we can assume that $b_3=1$.
 Also, the coefficient of $X_5$  can be vanished or be $\pm 1$
by $F^{s_6}_6$; By setting $s_6=-\mathrm{ln}\vert b_5 \vert$.
Then we have $\mathfrak{h}=\langle X_1,X_3+\varepsilon
X_5+X_6+b_7X_7 \rangle$ and $[X_1,X_3+\varepsilon
X_5+X_6+b_7X_7]=X_1$. So, $\mathfrak{h}$ is closed under the Lie
bracket and  the case $\mathcal{A}^4_1$ is deduced.
\item[h)] If $b_3=b_4=b_6=0$ and $b_2\neq 0$, then by scaling if necessary, we can assume that $b_2=1$. Also, we have  $\mathfrak{h}=\langle X_1,X_2+b_5X_5+b_7X_7 \rangle$ and $[X_1,X_2+b_5X_5+b_7X_7]=0$.
So, $\mathfrak{h}$ is closed under  the Lie bracket and the case
$\mathcal{A}^2_1$ is resulted.
\item[i)] If $b_4=b_6=0$, $b_2\neq 0$ and $b_3\neq 0$, then by scaling if necessary,
we can assume that $b_2=1$. Also we can make the coefficient of
$X_5$ vanish by $F^{s_5}_5$; By setting $s_5=-\frac{b_5}{b_3}$.
Then we have  $\mathfrak{h}=\langle X_1,X_2+b_3X_3+b_7X_7
\rangle$ and $[X_1,X_2+b_3X_3+b_7X_7]=b_3X_1$. So, $\mathfrak{h}$
is closed under the Lie bracket and the case $\mathcal{A}^5_1$ is
resulted .
\item[j)] If $b_4=0$, $b_6\neq 0$ and $b_3=1$, then by scaling if necessary, we can assume that $b_6=1$.
We can make the coefficient of $X_2$ vanish by $F^{s_2}_2$; By
setting $s_2=-\frac{1}{b_6}$. Also, the coefficient of $X_5$ can
be vanished or be $\pm 1$ by $F^{s_6}_6$; By setting
$s_6=-\mathrm{ln\vert b_5\vert}$. So, we have
$\mathfrak{h}=\langle X_1,X_3+\varepsilon X_5+X_6+b_7X_7 \rangle$
which is reduced to the case $\mathcal{A}^4_1$ similar to (g).
\item[k)] If $b_4=0$, $b_6\neq 0$ and $b_3\neq 1$, then we can make the coefficient of $X_2$ and $X_5$
vanish by $F^{s_2}_2$ and $F^{s_5}_5$; By setting
$s_2=-\frac{1}{b_6}$ and $s_5=-\frac{b_5}{b_3-1}$, respectively.
By scaling if necessary, we have  $\mathfrak{h}=\langle
X_1,X_3+b_6X_6+b_7X_7 \rangle$ that is reduced to the case
$\mathcal{A}^1_1$ similar to (f).
\end{itemize}
In each case, $\mathfrak{h}$  can not be simplified any more  by
$F^{s_i}_i$, $i=1,...,7$. By a similar method, we can find the
two-dimensional subalgebras for the other 11 cases.
\end{proof}

\section{ Computation of the partially invariant solutions for the system (\ref{2})}
In this section, we will calculate the PISs of the system (\ref{2}). For
example, consider the Lie subalgebra $\mathcal{A}^7_1 : \langle
X_1,X_7\rangle$.
 The set of functionally independent invariants for $\mathcal{A}^7_1$ is a set of functions $I$ with the following property: $X_1(I)=X_7(I)=0$.
By calculating these functions, a set of functionally independent
invariants is obtained as follows:
\begin{eqnarray}\label{16}
\{t,\frac{b}{a},\frac{c}{a}\}
\end{eqnarray}
So, we have:
\begin{equation}\label{17}
\mathrm{rank}\left(\frac{\partial(\frac{b}{a},\frac{c}{a})}{\partial(a,b,c)} \right)=\mathrm{rank}\left( \begin{array}{ccc}
   - b/a^2  & - c/a^2   \vspace*{1mm}\\
   1/a & 0  \vspace*{1mm}\\
   0 & 1/a \end{array}  \right)=2=q-\delta.
 \end{equation}
 Since $q=3$, then $\delta=1$. The equations corresponding  to the relation
(\ref{9}) are
 \begin{eqnarray}\label{18}
\begin{array}{lclcl}
\displaystyle \frac{b}{a}=f(t),&&\displaystyle\frac{c}{a}=g(t),
\end{array}
\end{eqnarray}
and the solutions corresponding to the equations (\ref{10}) and (\ref{11}) are
expressed as follows:
  \begin{eqnarray}\label{19}
\begin{array}{lclcl}
b=af(t),&&c=ag(t),&&a=a(x,t).
\end{array}
\end{eqnarray}
 Note that $a(x,t)$  is an arbitrary function.\\
 Now, we can  compute the derivatives of the functions $a,b$ and $c$ from relations (\ref{19}):
 \begin{eqnarray}\label{20}
\begin{array}{lclclclclclclc}
b_1=a_1f,&&\hspace{-2mm} b_2=a_2f+af',&&\hspace{-2mm} b_{11}=a_{11}f,&&\hspace{-2mm}b_{22}=a_{22}f+2a_2f'+af'',\\
 c_1=a_1g,&&\hspace{-2mm} c_2=a_2g+ag',&&\hspace{-2mm} c_{11}=a_{11}g,&&\hspace{-2mm}  c_{22}=a_{22}g+2a_2g'+ag'',\\
 b_{12}=a_{12}f+a_1f',&&\hspace{-2mm} c_{12}=a_{12}g+a_1g'.
\end{array}
\end{eqnarray}
Hence, by substituting the above relations into the system (\ref{2}), we
obtain:
\begin{eqnarray}\label{21}
\begin{array}{lclc}
a_{11}-a_{22}f-2a_2f'-af''=0,&&\hspace*{-65mm} a_{12}f+a_1f'+a_{11}g=0, \vspace*{1mm}\\
a_{12}+a_{22}g+2a_2g'+ag''=0,&&\hspace*{-65mm} a^2_2f+a_2af'-(a_2g+ag')^2+aa_{12}g+aa_{22}f=0,\vspace*{1mm}\\
a_1a_2f-a_1a_2g^2-2aa_1gg'-aa_{12}g^2-aa_{22}fg-2aa_2fg'-a^2fg''=0,\vspace*{1mm}\\
a^2_1f-2aa_1fg'-a^2_1g^2+aa_{11}f+3aa_1gf'+aa_{12}fg=0.\vspace*{1mm}
\end{array}
\end{eqnarray}
The consistency conditions, obtained from the system (\ref{21}),
conclude the following equations
 \begin{eqnarray}\label{22}
\begin{array}{lclc}
a^2f''+aa_2f'+a^2g'^2-a^2_2f+a^2_2g^2+2aa_2gg'=0,\hspace*{13mm} a_1=0,\vspace*{1mm}\\
aa_{22}f+a^2_2f-2aa_2gg'+aa_2f'-a^2{g'}^2-a^2_2g^2=0,\vspace*{1mm}\\
a^2fg''-aa_2f'g+a^2gg'^2-a^2_2fg+a^2_2g^3+2aa_2g'g^2+2aa_2fg'=0,\vspace*{1mm}
\end{array}
\end{eqnarray}
and these inequations
 \begin{eqnarray}
\begin{array}{lclclclc}\label{23}
f\neq 0,&& g\neq 0,&& a\neq 0,&& f-g^2\neq 0.
\end{array}
\end{eqnarray}
Equations (\ref{22}) form a system of ODEs. By solving this system, we
obtain four types of solutions as follows:
 \begin{equation}\label{24}
 \begin{array}{lclclclc}
1)\left \{ \begin{array}{lcr} f = c_3t+c_4\\ g = c_2  \\ a = c_1 \\ \end{array} \right.
 &&2)\left \{ \begin{array}{lcr} f = \frac{c_1c^2_3t+c_5}{c_1t+c_2}\vspace*{1mm}\\ g =c_3+\frac{c_1c_4}{c_1t+c_2} \vspace*{1mm} \\ a =c_1t+c_2 \\  \end{array} \right.\vspace*{1mm}
\\
3)\left \{ \begin{array}{lcr} f = \frac{c_5(t+c_2)}{\mathrm{ln}(t+c_2)-c_3c_1 } \vspace*{1mm}\\ g  = \frac{c_4}{\mathrm{ln}(t+c_2)-c_3c_1 } \vspace*{1mm} \\ a =- \frac{\mathrm{ln}(t+c_2)}{c_1 }+c_3 \\ \end{array} \right.
&& 4)\left \{ \begin{array}{lcr} f = \frac{c^2_6(t+c_2)}{(c_1\mathrm{ln}(t+c_2)+c_3t+c_4)c_3}\vspace*{1mm}\\ g =\frac{c_5+c_6t}{c_1\mathrm{ln}(t+c_2)+c_3t+c_4} \vspace*{1mm} \\ a =c_1\mathrm{ln}(t+c_2)+c_3t+c_4 \\  \end{array} \right.
\end{array}
\end{equation}

Now, by applying the relations (\ref{19}),  the partially invariant
solutions associated to the subalgebra $\mathcal{A}^7_1$  can be
obtained for the system (\ref{2}).
 \begin{equation}\label{25}
 \begin{array}{lclclclc}
1)\left \{ \begin{array}{lcr} a = c_1\\ b = c_1(c_3t+c_4)  \\ c=c_1c_2  \\ \end{array} \right.
&&2)\left \{ \begin{array}{lcr} a =c_1t+c_2 \\ b=c_1c^2_3t+c_5  \\ c =c_3(c_1t+c_2)+c_1c_4 \\  \end{array} \right.\vspace*{1mm}
\\
3)\left \{ \begin{array}{lcr} a =\displaystyle
\frac{\mathrm{-ln}(t+c_2)}{c_1 }+c_3 \vspace*{1mm}\\ b  =
c_5(t+c_2) \vspace*{1mm} \\ c=c_4 \\ \end{array} \right.
&&4)\left \{ \begin{array}{lcr} a
=c_1\mathrm{ln}(t+c_2)+c_3t+c_4 \vspace*{1mm}\\ b
=\displaystyle\frac{c^2_6(t+c_2)}{c_3} \vspace*{1mm} \\ c=c_5+c_6t \\
\end{array} \right.
\end{array}
\end{equation}
These partially invariant solutions seem to be trivial. In a similar
way, we can calculate nontrivial PISs by applying other
subalgebras listed in theorem (\ref{3.2}). Some other of these PISs are
presented in Table 1.
\section{Non-reducible partially invariant solutions}\vspace*{-1mm}
In this section, we will deal with those partially invariant
solutions which are not invariant with respect to some of the
subgroups of $G$. They are called non-reducible PISs and they are
usually rare. For a reducible partially invariant solution
$u=f(x)$, we can determine a subgroup $H'\subset H$ which  $u$ is
a $H'$-invariant solution and $\mathrm{dim}(H'\Gamma_f)\geqslant
\mathrm{dim}(H)-\delta=s-\delta$ (\cite{Ovs}, page 290). So,
reducible PISs can be obtained via the method of similarity
reduction from the reduced system involving
$p-\mathrm{dim}(H'\Gamma_f)\leqslant p+\delta-s$ independent
variables, which is indeed easier than obtaining them by PISs
method. For example, the PISs that we computed in the previous
section are non-reducible PISs.  Because,  reducible PISs  must be
invariant with respect to the one-parameter subgroups of its Lie
group. These subgroups have Lie subalgebras of $\mathcal{A}^7_1$
generated by an infinitesimal generator of the form
 $\alpha X_1+\beta X_7$ where $\alpha,\beta\in\mathbb{R}$.
In \cite{Nad}, we obtained the invariant solutions of the system
(\ref{2}) with respect to this generator by applying the method of
similarity reduction:
\begin{equation}\label{26}
1)\left \{ \begin{array}{lcr} a = 0\\ b = 0  \\ c=0 \\ \end{array} \right.
\hspace*{12mm}  2)\left \{ \begin{array}{lcr} a =0 \\ b=c_1e^{\frac{\beta}{\alpha}x}  \\ c =0 \\  \end{array} \right. \hspace*{12mm}
3)\left \{ \begin{array}{lcr} a = c_2e^{c_1t+\frac{\beta}{\alpha}x} \vspace*{1mm}\\ b=-\frac{c_2\beta}{c_1\alpha}e^{c_1t+\frac{\beta}{\alpha}x} \vspace*{1mm} \\ c=-\frac{c_1c_2\alpha}{\beta}e^{c_1t+\frac{\beta}{\alpha}x}  \\ \end{array} \right.
\end{equation}
Since, none of these solutions is similar to the PISs in (\ref{25}), we
conclude that PISs in (\ref{25}) are non-reducible PISs. Several other
non-reducible partially invariant solutions of the system (\ref{2}) are
listed in the following table.
\begin{center}
{\small \textbf{Table 1:} Non-reducible partially invariant solutions of (\ref{2}).}
\begin{small}
\begin{tabular}{lllllllll}
\cline{1-4}
Subalgebras & Invariants  & Dependent variables    & PISs    \\
\cline{1-4}
$\langle X_2,X_7\rangle$ & $\{x,\frac{b}{a},\frac{c}{a}\}$ & $\begin{array}{lcr} a = a(x,t)\vspace*{.7mm}\\ b = af(x) \vspace*{.7mm} \\ c=ag(x) \\ \end{array}$    & $\begin{array}{lcr} a = c_1x+c_2 \vspace*{.7mm} \\ b = c_1c^2_3x+(*)\vspace*{.7mm}  \\ c=c_3(c_1x+c_2)+c_1c_4 \\ \end{array}$  &\vspace*{3mm}\\

$\langle X_5,X_7\rangle$ & $\hspace*{-2mm}\begin{array}{lcr} \{t, \frac{c^2-ba}{b^2},\vspace*{1mm} \\
\frac{-ct+bx}{b}\}\end{array}$ & $\begin{array}{lcr} a = \frac{b(x-g(t))^2}{t^2}-bf(t)\vspace*{.7mm}\\ b = b(x,t) \vspace*{.7mm} \\ c=\frac{b}{t}(x-g(t)) \\ \end{array}$    & $\begin{array}{lcr} a = \frac{(c_1+c_2t^3)(x-c_3)^2}{t^3}-(**)\vspace*{.7mm}\\ b = \frac{c_1+c_2t^3}{t} \vspace*{.7mm} \\ c=\frac{(c_1+c_2t^3)(x-c_3)}{t^2} \\ \end{array}$  &\vspace*{3mm}\\

 $\langle X_2,X_6+X_7\rangle$ & $\{x,\frac{b}{a^3},\frac{c}{a^2}\}$ & $\begin{array}{lcr} a = a(x,t)\vspace*{.7mm}\\ b = a^3f(x)\vspace*{.7mm}  \\ c=a^2g(x) \\ \end{array}$    & $\begin{array}{lcr} a = \frac{4(t+c_1)}{(c_2x+c_3)^2}\vspace*{.7mm}\\ b =\frac{4c^2_2(t+c_1)^3}{(c_2x+c_3)^4} \vspace*{.7mm} \\ c=\frac{4c_2(t+c_1)^2}{(c_2x+c_3)^3} \\ \end{array}$  &\vspace*{3mm}\\

$\langle X_4,X_7\rangle$ &$\hspace*{-2mm}\begin{array}{lcr} \{x, \frac{-at+cx}{xa},\vspace*{1mm} \\
 \frac{at^2-2xtc+bx^2}{x^2a}\} \end{array}$   & $\begin{array}{lcr} a = a(x,t) \vspace*{.7mm}\\ b = a(g(x)+2\frac{t}{x}f(x)+\frac{t^2}{x^2}) \vspace*{.7mm} \\ c=a(f(x)+\frac{t}{x}) \\ \end{array}$    & $\begin{array}{lcr} a =\frac{c_1+c_2x^3}{x}\vspace*{.7mm} \\ b = \frac{(c_1+c_2x^3)(t+c_3)^2}{x^3}+(***) \vspace*{.7mm} \\ c=\frac{(c_1+c_2x^3)(t+c_3)}{x^2} \\ \end{array}$  &\vspace*{2mm}\\
\cline{1-4}
\end{tabular}
\end{small}
\end{center}
where $c_i$'s are arbitrary constants and
\begin{eqnarray*}
\begin{array}{lclc}
*=(\frac{c_5}{c_1}-c^2_3c_2)\mathrm{ln}(c_1x+c_2)+c_6 \vspace*{2mm},\\
**=t(c_4+c_5)(\mathrm{ln}(t+(\frac{c_1}{c_2})^{(\frac{1}{3})})^2-\mathrm{ln}(t^2-t(\frac{c_1}{c_2})^{(\frac{1}{3})}+(\frac{c_1}{c_2})^{(\frac{2}{3})})+2\sqrt{3}\mathrm{tan}^{-1}(\frac{2c_2t}{\sqrt{3}c_1}(\frac{c_1}{c_2})^{(\frac{2}{3})}-\frac{1}{\sqrt{3}})),\vspace*{2mm}\\
***=x(c_4+c_5)(\mathrm{ln}(x+(\frac{c_1}{c_2})^{(\frac{1}{3})})^2-\mathrm{ln}(x^2-x(\frac{c_1}{c_2})^{(\frac{1}{3})}+(\frac{c_1}{c_2})^{(\frac{2}{3})})+2\sqrt{3}\mathrm{tan}^{-1}(\frac{2c_2x}{\sqrt{3}c_1}(\frac{c_1}{c_2})^{(\frac{2}{3})}-\frac{1}{\sqrt{3}})).
\end{array}
\end{eqnarray*}
\section*{Conclusion}
Partially invariant solutions can be regarded as the natural
extension of the invariant ones and the method of obtaining them
is algorithmic. Some of the PISs are not invariant with respect to
the lower dimensional subalgebras. They are called non-reducible
PISs and  can be applied in order to obtain some new solutions for
an arbitrary system of PDEs. For example, in the present paper, we
have obtained a non-reducible PISs for the system (\ref{2}) as follows:
\begin{eqnarray}\label{27}
\begin{array}{lclclc}
a = \frac{4(t+c_1)}{(c_2x+c_3)^2},&& b
=\frac{4c^2_2(t+c_1)^3}{(c_2x+c_3)^4},&&
c=\frac{4c_2(t+c_1)^2}{(c_2x+c_3)^3},
\end{array}
\end{eqnarray}
Consequently, according to relation (\ref{1}), we could determine a general form
for the metric associated to a set of four-dimensional Einstein
Walker manifolds as follows:
\begin{eqnarray}\label{28}
\begin{array}{lclclc}
 g_{a,b,c}:=2(dx\circ dy+dt\circ  dz)+\displaystyle\frac{4(t+c_1)}{(c_2x+c_3)^2}dy\circ dy\vspace*{2mm}\\
 \hspace{13mm}+\displaystyle\frac{4c^2_2(t+c_1)^3}{(c_2x+c_3)^4}dz\circ dz+\displaystyle\frac{8c_2(t+c_1)^2}{(c_2x+c_3)^3}dy\circ dz.
\end{array}
\end{eqnarray}
This set of  solutions can not be obtained from the similarity
reduction method. In a similar manner, we can deduce some other
four-dimensional Einstein Walker manifolds from Table 1.

\section*{Acknowledgements}
It is a pleasure to thank the anonymous referees for their
constructive suggestions and helpful comments which have improved
the presentation of the paper. The authors wish to express
their sincere gratitude to Fatemeh Ahangari for her useful advise
and suggestions.







\end{document}